\documentclass[11pt]{amsart}
\usepackage[utf8x]{inputenc}
\usepackage[english]{babel}
\usepackage[T1]{fontenc}
 
\usepackage{graphicx}
\usepackage{caption}
\usepackage{subcaption}

\usepackage{amssymb,amsmath}
\usepackage[hidelinks]{hyperref}
\usepackage{indentfirst}
\usepackage{enumerate,amsmath,amssymb, mathrsfs,mathtools}
\usepackage{appendix}
\usepackage{latexsym}
\usepackage{url}
\usepackage{color}
\usepackage{accents}\usepackage[utf8x]{inputenc}
\usepackage[english]{babel}
\usepackage[T1]{fontenc}
 
\usepackage{graphicx}
\usepackage{caption}
\usepackage{subcaption}

\usepackage{amssymb,amsmath}
\usepackage[hidelinks]{hyperref}
\usepackage{indentfirst}
\usepackage{enumerate,amsmath,amssymb, mathrsfs,mathtools}
\usepackage{appendix}
\usepackage{latexsym}
\usepackage{url}
\usepackage{color}
\usepackage{accents}
\usepackage{setspace}
\usepackage{pdfpages}
\usepackage{stmaryrd}
\usepackage{amsrefs}

\usepackage[margin=2.5cm]{geometry}
\allowdisplaybreaks

\newcommand{\hcirc}{\accentset{\circ}{h}}

\newtheorem{prop}{Proposition}
\newtheorem{thm}[prop]{Theorem}
\newtheorem{lem}[prop]{Lemma}

\newtheorem{rema}[prop]{Remark}

\newtheorem{RK}[prop]{Remark}

\usepackage{setspace}
\usepackage{pdfpages}
\usepackage{stmaryrd}
\usepackage{amsrefs}

\usepackage[margin=2.5cm]{geometry}
\allowdisplaybreaks

\BibSpec{article}{%
	+{}{\PrintAuthors} {author}
	+{,}{ } {title}
	+{, }{\textit } {journal}
	+{}{ \parenthesize} {date}
		+{,  }{no. } {volume}
	+{,}{ } {pages}
	+{,}{ } {note}
}

\ExplSyntaxOn

\ExplSyntaxOff
\BibSpec{book}{%
	+{}{\PrintAuthors}  {author}
	+{. }{}{title}
	+{,}{ }{series}
	+{,}{ vol.~}{volume}
	+{. }{\textit}{publisher}
	+{,}{ ISBN}{isbn} %{ISBN \gobblefirst}{isbn}
}
\BibSpec{collection.article}{%
	+{}{\PrintAuthors}{author}
	+{, }{}{title}
	+{, }{\textit}{booktitle}
	+{, }{ \DashPages}{pages}
	+{,}{ }{series}
	+{, }{}{volume}
	+{, }{\textit}{publisher}
	+{,}{ }{date}
}

\mathcode`l="8000
\begingroup
\makeatletter
\lccode`\~=`\l
\DeclareMathSymbol{\lsb@l}{\mathalpha}{letters}{`l}
\lowercase{\gdef~{\ifnum\the\mathgroup=\m@ne \ell \else \lsb@l \fi}}%
\endgroup

\def\XXint#1#2#3{{\setbox0=\hbox{$#1{#2#3}{\int}$ }
		\vcenter{\hbox{$#2#3$ }}\kern-.6\wd0}}

\title[Inverse mean curvature flow and  Ricci-pinched  three-manifolds]{Inverse mean curvature flow and  Ricci-pinched  three-manifolds}

\author{Gerhard Huisken}
\address{\textnormal{Gerhard Huisken  \newline \indent
		University of Tübingen \newline \indent
		Department of Mathematics  \newline \indent
	Auf der Morgenstelle 10  \newline \indent 72076 Tübingen,	Germany\newline\indent 
		\href{mailto:gerhard.huisken@uni-tuebingen.de}{gerhard.huisken@uni-tuebingen.de}}
}

\author{Thomas Koerber}
\address{\textnormal{Thomas Koerber  \newline \indent
		University of Vienna \newline \indent
		Faculty of Mathematics  \newline \indent
		Oskar-Morgenstern-Platz 1 \newline \indent 1090 Vienna,	Austria \newline\indent 
		 \href{https://orcid.org/0000-0003-1676-0824}{https://orcid.org/0000-0003-1676-0824} \newline \indent
		  \href{mailto:thomas.koerber@univie.ac.at}{thomas.koerber@univie.ac.at}}
}

\begin{document}
\maketitle
\begin{abstract}
	Let $(M,g)$ be a noncompact, connected, complete Riemannian three-manifold with nonnegative Ricci curvature satisfying $Ric\geq\varepsilon\,\operatorname{tr}(Ric)\,g$ for some $\varepsilon>0$. In this note, we give a new proof based on inverse mean curvature flow that $(M,g)$ is either flat or has non-Euclidean volume growth. In conjunction with the work of J.~Lott \cite{lott} and of M.-C.~Lee and P.~Topping \cite{leetopping}, this gives an alternative proof of a conjecture of R.~Hamilton recently proven by A.~Deruelle, F.~Schulze, and M.~Simon \cite{deruelleschulzesimon} using Ricci flow.
\end{abstract}
	\date{\today}
	\onehalfspacing
	\section{Introduction}
		Let $(M,g)$ be a noncompact, connected, complete Riemannian three-manifold with nonnegative Ricci curvature.  Recall that $(M,g)$ is called Ricci-pinched if  there is $\varepsilon>0$ such that 
		\begin{align} \label{pinching} 
			Ric\geq \varepsilon \,R\, g.
		\end{align} 
	Here, $Ric$ and $R$ denote the Ricci curvature and the scalar curvature of $(M,g)$, respectively. The following theorem has been conjectured by R.~Hamilton 	\cite[Conjecture 3.39]{chowluni} and by J.~Lott \cite[Conjecture 1.1]{lott}. It has been proven by A.~Deruelle, F.~Schulze, and M.~Simon  \cite[Theorem 1.3]{deruelleschulzesimon} under the additional assumption that $(M,g)$ has bounded curvature. M.-C.~Lee and P.~Topping \cite[Theorem 1.2]{leetopping} have subsequently shown that this additional assumption can be dispensed with.
	\begin{thm}[{\cite[Theorem 1.2]{leetopping}}] \label{main result 0}
		Let $(M,g)$ be a noncompact, connected, complete Riemannian three-manifold that is Ricci-pinched. Then $(M,g)$ is flat.  
	\end{thm}
\begin{rema}
Previous results in the direction of	 Theorem \ref{main result 0} have been obtained by B.-L.~Chen and X.-P.~Zhu \cite[Main Theorem II]{chenzhu} and by J.~Lott
\cite[Theorem 1.4]{lott}.
\end{rema}
\begin{rema}
R.~Hamilton has shown that every compact, connected, complete  Riemannian three-manifold that is Ricci-pinched is either flat or smoothly isotopic to a spherical space form; see \cite[Main Theorem]{hamilton} and the discussion on \cite[p.~4]{deruelleschulzesimon}.
\end{rema}
To describe the contributions in \cite{lott,deruelleschulzesimon,leetopping}, let $p\in M$ and
$$
\operatorname{AVR}=\frac{3}{4\,\pi}\,\lim_{r\to\infty} \frac{|B_r(p)|}{r^3}
$$
be the  asymptotic volume ratio of $(M,g)$. Note that, by the Bishop-Gromov theorem, $\operatorname{AVR}$ is well-defined, independent of $p$, and satisfies $\operatorname{AVR}\in[0,1]$. J.~Lott \cite{lott} has shown that if $(M,g)$ is Ricci-pinched and has bounded curvature,  then there exists a smooth, Ricci-pinched solution of Ricci flow coming out of $(M,g)$. By performing a detailed asymptotic analysis of this flow, J.~Lott has proven that, if $(M,g)$ is not flat, then $(M,g)$ has positive asymptotic volume ratio; see \cite[Corollary 1.7]{lott}. Subsequently, M.-C.~Lee and P.~Topping \cite{leetopping} have shown that the assumption that $(M,g)$ has bounded curvature can be dispensed with. By contrast, if $(M,g)$ has positive asymptotic volume ratio, A.~Deruelle, F.~Schulze, and M.~Simon \cite[Lemma 8.2]{deruelleschulzesimon} have observed that the asymptotic cone of $(M,g)$ is a three-dimensional Alexandrov space with nonnegative curvature. Moreover, the previous work of A.~Deruelle \cite{Deruelle} respectively of F.~Schulze and M.~Simon \cite{SchulzeSimon} implies the existence  of an expanding
soliton solution with nonnegative curvature coming out of  the asymptotic  cone of $(M,g)$.  Using their stability result \cite[Theorem 1.2]{deruelleschulzesimon}  to compare this solution with the Ricci-pinched solution constructed by J.~Lott \cite{lott}, they have concluded that $(M,g)$ is in fact flat.
 \\ 
\indent  
 The goal of this paper is to provide a new  proof based on inverse mean curvature flow of Theorem \ref{main result 0} under the additional assumption that $(M,g)$ has positive asymptotic volume ratio.  
 \begin{thm} \label{main result} 
 	Let $(M,g)$ be a noncompact, connected, complete Riemannian three-manifold that is Ricci-pinched. If $(M,g)$ has positive asymptotic volume ratio, then $(M,g)$ is isometric to flat $\mathbb{R}^3$.
 \end{thm}
 \begin{rema}
In conjunction with the results of J.~Lott \cite{lott} and of M.-C.~Lee and P.~Topping \cite{leetopping}, Theorem \ref{main result} gives a new proof of Theorem \ref{main result 0}.
 \end{rema}
\begin{rema}
Our technique extends to the case where	$(M,g)$ has superquadratic volume growth and the volume of geodesic unit balls in $(M,g)$ is noncollapsed; see Theorem \ref{main result 2}. 
\end{rema}
	We now outline the proof of Theorem \ref{main result}. Suppose, for a contradiction, that $(M,g)$ is not flat. Since $(M,g)$ has nonnegative Ricci curvature and is Ricci-pinched,  the scalar curvature of $(M,g)$ must be strictly positive at one point. It follows that there is a closed, connected, outward-minimizing surface $\Sigma\subset M$ such that
	\begin{align} \label{intro initial willmore bound}
	\int_\Sigma H^2\,\mathrm{d}\mu<16\,\pi.
	\end{align} 
	Here, $\mathrm{d}\mu$ and $H$ denote the area element and the mean curvature of $\Sigma$, respectively. Recall that a nested family $\{E_t\}_{t=0}^\infty$ of precompact, open sets $E_t\subset M$ with smooth and strictly mean-convex boundary $\partial E_t$ flows by inverse mean curvature flow if 
	$$
	\frac{dx}{dt}=H^{-1}\,\nu.
	$$
	Here, $x$ and $\nu$ are the position and the  outward normal of $\partial E_t$, respectively.  Using that $(M,g)$ has positive asymptotic volume ratio, the work of K.~Xu \citeauthor{xu2023isoperimetry} shows that there exists a weak solution $\{E_t\}_{t=0}^\infty$ of inverse mean curvature flow  with $\partial^*E_0=\Sigma$ in the sense of the work of T.~Ilmanen and the first-named author \cite{huiskenilmanen}.  Here, $\partial^*$ denotes the reduced boundary. Using \eqref{intro initial willmore bound} and that $(M,g)$ is Ricci-pinched, we show that
	$$
	\lim_{t\to\infty} \int_{\partial^* E_t}H^2\,\mathrm{d}\mu=0. 
	$$ 
	By contrast, the work of V.~Agostiniani, M.~Fogagnolo, and L.~Mazzieri \cite{agostinianifogagnolomazzierei} implies that, for every $t\geq 0$,
	\begin{align} \label{willmore lower bound} 
	\int_{\partial^* E_t}H^2\,\mathrm{d}\mu\geq 16\,\pi\,\operatorname{AVR},
	\end{align} 
	a contradiction; see also the work of X.~Wang \cite{wang} for an alternative proof of \eqref{willmore lower bound}.
\subsection*{Acknowledgments} The second-named author acknowledges the support of the Lise-Meitner-Project M3184 of the Austrian Science Fund. This work originated during the authors’ visit to the Hebrew University during the first-named author's Mark Gordon Distinguished Visiting Professorship. The authors thank the Hebrew University for its hospitality. The authors thank Or Hershkovits, Marco Pozetta, and Miles Simon for helpful discussions. The authors thank the referees for the valuable comments and corrections.
\section{Proof of Theorem \ref{main result}}
In this section, we assume that $(M,g)$ is a noncompact, connected, complete Riemannian three-manifold with nonnegative Ricci curvature satisfying \eqref{pinching} for some $\varepsilon>0$.\\ \indent
The goal of this section is to prove Theorem \ref{main result}.
\\ \indent We recall the following result of S.-H.~Zhu \cite{zhu}, which extends previous work of R.~Schoen and S.-T.~Yau \cite[Theorem 3]{SchoenYau}.
\begin{prop} \label{toplogy not flat} 
	 If $(M,g)$ is not flat, then $(M,g)$ is diffeomorphic to $\mathbb{R}^3$.
\end{prop} 
\begin{proof}
 Using \eqref{pinching}, we see that there is $p\in M$ with $Ric(p)>0$. The assertion follows from \cite{zhu}.
\end{proof}
	Let $\Sigma\subset M$ be a closed, connected surface with area measure $\mathrm{d}\mu,$  designated normal $\nu$, and mean curvature $H$ with respect to $\nu$. In Lemma \ref{willmore est} below, $\hcirc$ denotes the traceless second fundamental form of $\Sigma$.
\begin{lem} \label{willmore est} 
If  $\operatorname{genus}(\Sigma)\geq 1$, there holds
	\begin{align} \label{genus 1}  
	2\,\int_{\Sigma}Ric(\nu,\nu)+|\hcirc|^2\,\mathrm{d}\mu\geq \int_{\Sigma} H^2\,\mathrm{d}\mu 
	\end{align} 
	 and, if $\operatorname{genus}(\Sigma)=0$, there holds
	\begin{align} \label{genus 0}  
		2\,\int_{\Sigma}Ric(\nu,\nu)\,\mathrm{d}\mu\geq\varepsilon\,\bigg(16\,\pi-\int_{\Sigma}H^2\,\mathrm{d}\mu\bigg).
	\end{align}  
\end{lem}
\begin{proof}
	Integrating the contracted Gauss equation and using the Gauss-Bonnet theorem, we have
	$$
	\int_{\Sigma} H^2\,\mathrm{d}\mu=16\,\pi\,(1-\operatorname{genus}(\Sigma))+2\,\int_{\Sigma} |\hcirc|^2\,\mathrm{d}\mu+\int_{\Sigma}4\,Ric(\nu,\nu)-2\,R\,\mathrm{d}\mu. 
	$$
	\indent Using that $Ric\geq 0$, we have $R\geq Ric(\nu,\nu)$. This implies \eqref{genus 1}. \\ \indent In the case where $\operatorname{genus}(\Sigma)=0$, we have, using that $Ric \geq 0$, 
	$$
	2\,\int_{\Sigma}R\,\mathrm{d}\mu\geq 16\,\pi-\int_{\Sigma} H^2.
	$$
	In conjunction with \eqref{pinching}, we obtain \eqref{genus 0}. 
\end{proof}
\begin{lem}  \label{small willmore} 
	Suppose that $(M,g)$ is not flat. There exists a sequence $\{\Sigma_i\}_{i=1}^\infty$ of closed, connected surfaces $\Sigma_i\subset M$ with 
	$$
\lim_{i\to\infty} 	\int_{\Sigma_i}H^2\,\mathrm{d}\mu=0.
	$$
\end{lem}
\begin{proof}
	In the case where $\operatorname{AVR}=0$, the assertion follows from \cite[Theorem 1.1]{agostinianifogagnolomazzierei}. \\ \indent 
	In the case where $\operatorname{AVR}>0$, using that $(M,g)$ is not flat and that $Ric\geq 0$, we see that there is $p\in M$ with $R(p)>0$. Consequently, 
	\begin{align} \label{initial gap} 
	\int_{B_r(p)} H^2\,\mathrm{d}\mu<16\,\pi 
	\end{align} 
	provided that $r>0$ is sufficiently small; see, e.g.,~\cite[Proposition 3.1]{mondino}. Let $\Sigma'\subset M$ be the outward-minimizing hull of $B_r(p)$; see \cite[p.~371]{huiskenilmanen}. Using \cite[(1.15)]{huiskenilmanen}, we see that
	$$
	\int_{\Sigma'}H^2\,\mathrm{d}\mu\leq \int_{B_r(p)}H^2\,\mathrm{d}\mu.
	$$
By the Bishop-Gromov theorem, we have
$$
|B_1(q)|\geq \frac{4\,\pi}{3}\,\operatorname{AVR}
$$	
for every $q\in M$.   Using this and that $\operatorname{AVR}>0$, the work of T.~Coulhon and L.~Saloff-Coste \cite{CoulhonSaloff-Coste} implies that there is $\gamma>0$ such that
$ 
	|\partial \Omega|\geq \gamma\,|\Omega|^{2/3}
$ 
for every precompact, open set $\Omega\subset M$  with smooth boundary $\partial \Omega$;  see the remarks below \cite[Theor\`eme 3]{CoulhonSaloff-Coste} and also \cite{agostinianifogagnolomazzierei,Brendle}.
In conjunction with \cite[Theorem 1.2]{xu2023isoperimetry},  it follows that there exists a proper weak solution $\{E_t\}_{t=0}^\infty$ of inverse mean curvature flow  in the sense of \cite[p.~368]{huiskenilmanen} such that $\Sigma'=\partial^* E_0$. Let $\Sigma_t=\partial^* E_t$. By Proposition \ref{toplogy not flat} and \cite[Lemma 4.2]{huiskenilmanen}, $\Sigma_t$ is connected for every $t\geq 0$. According to the results in \cite[\S5]{huiskenilmanen} and \cite[Korollar 5.6]{Heidusch}, $\Sigma_t$ is of class $W^{2,2}\cap C^{1,1}$ and there holds
	\begin{align} \label{monotonicity}  
	\int_{\Sigma_0} H^2\,\mathrm{d}\mu\geq \int_{\Sigma_t}H^2\,\mathrm{d}\mu+2\,\int_0^t\int_{\Sigma_s}Ric(\nu,\nu)+|\hcirc|^2\,\mathrm{d}\mu\,\mathrm{d}s
	\end{align} 
	for every $t\geq 0$. Clearly, the function \begin{align}  \label{Willmore energy}
	[0,\infty)\to\mathbb{R},\qquad  t\mapsto \int_{\Sigma_t} H^2\,\mathrm{d}\mu
	\end{align} 
	is nonincreasing. We claim that $$
	\lim_{t\to\infty} \int_{\Sigma_t} H^2\,\mathrm{d}\mu=0.
	$$
	Indeed, suppose, for a contradiction, that there is $\delta>0$ such that
	$$
	\int_{\Sigma_t} H^2\,\mathrm{d}\mu\geq \delta 
	$$ 
	for every $t\geq 0$. Shrinking $\delta>0$, if necessary, and using \eqref{initial gap}, we may assume that 
	$$
	\int_{\Sigma_t} H^2\,\mathrm{d}\mu \leq 16\,\pi-\delta 
	$$
	for every $t\geq0$. Using Lemma \ref{willmore est}, we have
	$$
	2\,\int_{\Sigma_s} Ric(\nu,\nu)+|\hcirc|^2\,\mathrm{d}\mu \geq \min\{1,\varepsilon\}\,\delta 
	$$
	for every $t\geq 0$. In conjunction with \eqref{monotonicity}, we see that
	$$
	\int_{\Sigma_t}H^2\,\mathrm{d}\mu<0
	$$ 
	for every $t\geq 16\,\pi\,\min\{1,\varepsilon\}^{-1}\,\delta^{-1}$. The assertion follows from this contradiction.
\end{proof}
\begin{proof}[Proof of Theorem \ref{main result}]
	Let $(M,g)$ be a noncompact, connected, complete  Riemannian three-manifold  that is Ricci-pinched and has positive asymptotic volume ratio $\operatorname{AVR}$. Suppose, for a contradiction, that $(M,g)$ is not flat $\mathbb{R}^3$. By Lemma \ref{small willmore}, there exists a closed, connected surface $\Sigma\subset M$ with
	$$
	\int_{\Sigma} H^2\,\mathrm{d}\mu<16\,\pi\,\operatorname{AVR}.
	$$ 
	As this is incompatible with \cite[Theorem 1.1]{agostinianifogagnolomazzierei}, the assertion follows.
\end{proof}
\begin{appendices}

\section{Superquadratic volume growth}
In this section, we assume that $(M,g)$ is a noncompact, connected, complete Riemannian three-manifold with nonnegative Ricci curvature satisfying \eqref{pinching} for some $\varepsilon>0$. Moreover, we assume that there is $p\in M$ and $\alpha>0$ with 
\begin{align} \label{volume growth}  
\liminf_{r\to\infty} \frac{|B_r(p)|}{r^{1+\alpha}}>0
\end{align} 
and that
\begin{align} \label{non-collapsed} 
\inf_{q\in M}|B_1(q)|>0.
\end{align} 
Note that, by the Bishop-Gromov theorem, $\alpha\leq 2$. Moreover, note that if $\alpha\leq 1$, then $(M,g)$ is parabolic; see \cite[Corollary 2.3]{LiTam}. \\ \indent
The goal of this section is to give an alternative proof, based on inverse mean curvature flow, of the fact that  $(M,g)$ is either flat or has subquadratic volume growth, that is, $\alpha\leq 1$. 
 \\ \indent 
 By the work of T.~Coulhon and L.~Saloff-Coste, using \eqref{volume growth} and \eqref{non-collapsed}, the following isoperimetric inequality
 \begin{align}  \label{isocoulhon}
 |\partial \Omega|\geq \gamma\,\min\{|\Omega|^{2/3},|\Omega|^{\alpha/(1+\alpha)}\}
\end{align} 
holds for some $\gamma>0$ and every precompact open set $\Omega\subset M$ with smooth boundary $\partial \Omega$; see the remarks below \cite[Theor\`eme 3]{CoulhonSaloff-Coste}.
\begin{lem} \label{rapid decay} 
	Suppose that $(M,g)$ is not flat. Then there exists a proper weak solution $\{E_t\}_{t=0}^\infty$ of inverse mean curvature flow  such that, for every $t\geq 0$,
	$$
	\int_{\partial^* E_t} H^2\,\mathrm{d}\mu < 16\,\pi\,e^{-t} . 
	$$
	
\end{lem}
\begin{proof}
	By Lemma \ref{small willmore}, there is a closed, connected surface $\Sigma\subset M$ satisfying 
	$$
	\varepsilon\,(16\,\pi-\int_{\Sigma}H^2\,\mathrm{d}\mu)\geq \int_{\Sigma} H^2\,\mathrm{d}\mu. 
	$$ 
Replacing $\Sigma$ by its outward-minimizing hull and using \cite[(1.15)]{huiskenilmanen}, we may assume that $\Sigma$ is outward-minimizing. By \cite[Theorem 1.2]{xu2023isoperimetry}, using \eqref{isocoulhon},	there is a proper weak solution $\{E_t\}_{t=0}^\infty$ of inverse mean curvature flow with ${\partial^*} E_0=\Sigma$. Let $\Sigma_t=\partial^*E_t$. As in the proof of Lemma \ref{small willmore}, we see that the function \eqref{Willmore energy} is nonincreasing so that, for every $t\geq 0$,  
	\begin{align} \label{varepsilon Willmore} 
\varepsilon\,(16\,\pi-\int_{\Sigma_t}H^2\,\mathrm{d}\mu)\geq \int_{\Sigma_t} H^2\,\mathrm{d}\mu. 
\end{align} 
 Again, as in the proof of Lemma \ref{small willmore}, using \eqref{varepsilon Willmore} and Lemma \ref{willmore est}, we see that
	$$
		\int_{\Sigma} H^2\,\mathrm{d}\mu\geq \int_{\Sigma_t}H^2\,\mathrm{d}\mu+\int_0^t\int_{\Sigma_s}H^2\,\mathrm{d}\mu\,\mathrm{d}s.
	$$
 	Using that
	$$\int_{\Sigma} H^2\,\mathrm{d}\mu<16\,\pi, $$
the assertion follows. 
\end{proof}
\begin{RK}
	The existence of a proper weak solution of inverse mean curvature flow as asserted in Lemma \ref{rapid decay} would also follow from \cite[Remark 1.6 and Theorem 1.7]{maririgolisetti}, even without assuming \eqref{non-collapsed}, but replacing \eqref{volume growth} by 
	$$
	0<\lim_{r\to\infty} \frac{|B_r(p)|}{r^{1+\alpha}}<\infty.
	$$
	 However, we were not able to verify all arguments
	in the proof of \cite[Theorem 3.6]{maririgolisetti}.
\end{RK}
\begin{thm} \label{main result 2} 
	Let $(M,g)$ be a noncompact, connected, complete Riemannian three-manifold that is Ricci-pinched and satisfies \eqref{volume growth} for some $\alpha>0$ and \eqref{non-collapsed}. If $(M,g)$ is not flat, then $\alpha\leq 1$.
\end{thm}
\begin{proof}
	Let $\{E_t\}_{t=0}^\infty$ be a proper weak solution of  inverse mean curvature flow as in Lemma \ref{rapid decay}. Let $\Sigma_t=\partial^* E_t$. By  \cite[Lemma 5.1]{huiskenilmanen}, there holds $H>0$ $\mathrm{d}\mu$-almost everywhere on $\Sigma_t$ and $H^{-1}\in L^1(\Sigma_t)$ for almost every $t\geq 0$.  Using Hölder's inequality, we have, for almost every $t\geq 0$, 
	\begin{align} \label{hoelder estimate}
	|\Sigma_t|^{3/2}\leq \int_{\Sigma_t}H^{-1}\,\mathrm{d}\mu\,\left(\int_{\Sigma_t} H^2\,\mathrm{d}\mu \right)^{1/2}.
	\end{align} 
	Recall from \cite[Exponential Growth Lemma 1.6]{huiskenilmanen}, that \begin{align} \label{exponential area growth} |\Sigma_t|=|\Sigma_0|\,e^t. \end{align}
 In conjunction with Lemma \ref{rapid decay} and \eqref{hoelder estimate}, we obtain that, for almost every $t\geq 0$,
	\begin{align} \label{volume speed estimate}  
	\int_{\Sigma_t} H^{-1}\,\mathrm{d}\mu\geq \sqrt{\frac{|\Sigma_0|^3}{16\,\pi}}\,e^{2\,t}.
	\end{align} 
Arguing as in \cite[\S5]{huiskenilmanen}, we see that
	$$
	|E_t|\geq |E_0|+\int_0^t\,\int_{\Sigma_s}H^{-1}\,\mathrm{d}\mu\,\mathrm{d}s 
	$$
	for every $t\geq 0$. In conjunction with \eqref{exponential area growth} and \eqref{volume speed estimate},	it follows that 
	\begin{align} \label{volume area comparison}  
	\liminf_{t\to\infty} |\Sigma_t|^{-2}\,|E_t|>0.
	\end{align} 
By contrast, using \eqref{isocoulhon}, we have
$$
|\Sigma_t|^{-2}\,|E_t|\leq \gamma ^{-2}\,|E_t|^{(1-\alpha)/(1+\alpha)}.
$$
As this is incompatible with \eqref{volume area comparison} unless $\alpha\leq 1$, the assertion follows. 
\end{proof} 

\end{appendices}

\end{document}